\documentclass[11pt]{amsart}
\usepackage{amsmath,amstext,amsbsy,amssymb,amscd}
\usepackage{amsmath}
\usepackage{amsxtra}
\usepackage{amscd}
\usepackage{amsthm}
\usepackage{amsfonts}
\usepackage{graphicx}%
\usepackage{amssymb}
\usepackage{eucal}
\usepackage{color}
\usepackage{multirow}
\usepackage[enableskew,vcentermath]{youngtab}

\newtheorem{thm}{Theorem}[section]
\theoremstyle{plain}

\newtheorem{lem}[thm]{Lemma}
\newtheorem{prop}[thm]{Proposition}
\newtheorem{cor}[thm]{Corollary}

\theoremstyle{definition}

\newtheorem{example}[thm]{Example}

\theoremstyle{remark}
\newtheorem{rem}[thm]{Remark}

\newtheorem*{thma}{{\bf Theorem A}}
\newtheorem*{thmb}{{\bf Theorem B}}
\newtheorem*{thmc}{{\bf Theorem C}}
\definecolor{A}{rgb}{.75,1,.75}

\numberwithin{equation}{section}

\newcommand{\ds}{\displaystyle}

\newcommand{\C}{\mathbb C}
\newcommand{\Z}{\mathbb Z}
\newcommand{\N}{\mathbb N}

\newcommand{\Cl}{\mathcal{C}_n}
\newcommand{\HC}{\mathcal{H}_n}
\newcommand{\la}{\lambda}
\newcommand{\spin}{\C S_n^-}

{\vskip-\lastskip\medskip
  \noindent
  {\em #1.}\enspace
  }%
{\qed\par\medskip
  }

\begin{document}

\title[Spin invariant theory for the symmetric group]{Spin invariant theory for the symmetric group}
\author[Wan and Wang]{Jinkui Wan and Weiqiang Wang}
\address{
Department of Mathematics,
Beijing Institute of Technology,
Beijing, 100081, P.R. China. }
\email{wjk302@gmail.com}

\address{Department of Mathematics, University of Virginia,
Charlottesville,VA 22904, USA.}
\email{ww9c@virginia.edu}
\begin{abstract}
We formulate a theory of invariants for the spin symmetric group
in some suitable modules which involve the polynomial and exterior
algebras. We solve the corresponding graded multiplicity problem
in terms of specializations of the Schur $Q$-functions and a
shifted $q$-hook formula. In addition, we provide a bijective
proof for a formula of the principal specialization of the Schur
$Q$-functions.
\end{abstract}

\maketitle


\section{Introduction}

\subsection{}
\label{subsec:Sn}

The symmetric group $S_n$ acts on $V=\C^n$ and then on the
symmetric algebra $S^*V$ naturally. It is well known that the
algebra of $S_n$-invariants on $S^*V$ is a polynomial algebra in
$n$ generators of degree $1,2, \ldots, n$. More generally,
consider the graded multiplicity of a given Specht module $S^\la$
for a partition $\la=(\la_1, \la_2, \ldots)$ of $n$ in the graded
algebra $S^*V$, which has a generating function $P_\la (t):
=\sum_{j\geq 0} m_\la (S^jV) t^j$. Kirillov \cite{Ki} has obtained
the following elegant formula for $P_\la (t)$
(also compare Steinberg~\cite{S}):
\begin{align*}
P_\la (t) =\frac{t^{n(\la)}} {\prod_{(i,j)\in \la}(1-t^{h_{ij}})},
\end{align*}
where $h_{ij}$ is the hook length associated to a cell $(i,j)$ in
the Young diagram of $\la$, and
\begin{eqnarray}   \label{n lambda}
n(\lambda)=\sum_{i\geq 1}(i-1)\lambda_i.
\end{eqnarray}

The generating function for the bi-graded $S_n$-invariants in
$S^*V \otimes \wedge^* V$ was computed in Solomon \cite{So}, see
\eqref{Solomon}. More generally, Kirillov and Pak \cite{KP}
obtained the bi-graded multiplicity of the Specht module $S^\la$
for any $\la$ in $S^*V \otimes \wedge^* V$, see \eqref{eq:KPak}.

\subsection{}

According to Schur \cite{Sch}, the symmetric group $S_n$ affords a
double cover $\widetilde{S}_n$:
$$
1 \longrightarrow \Z_2 \longrightarrow \widetilde{S}_n
\longrightarrow S_n \longrightarrow 1.
$$
Let us write $\Z_2 =\{1,z\}$. The spin (or projective)
representation theory of $S_n$, or equivalently, the
representation theory of the spin group algebra $\C S_n^- =\C
\widetilde{S}_n/\langle z+1\rangle$, has been systematically
developed by Schur (see J\'ozefiak \cite{Jo1} for an excellent
modern exposition via a systematic use of superalgebras). Rich
algebraic combinatorics of Schur $Q$-functions and shifted
tableaux has been developed by Sagan \cite{Sa1} and Stembridge
\cite{St} (also see Nazarov \cite{Naz}) in relation to the
irreducible spin representations and characters of $S_n$.

\subsection{}

The goal of this paper is to formulate and prove the spin analogue
of the graded multiplicity formulas in \ref{subsec:Sn}.

The results of this paper, though looking classical, have not
appeared in literature to our knowledge; however, it is expected
that such simple results, once formulated, can be also derived by
other approaches. It strongly suggests that the spin invariant
theory of Weyl groups, or of finite groups in general, in the
sense of this paper is a very interesting research direction to
pursue. It is also natural to ask for the spin double counterpart,
a spin analogue of Kostka polynomials, an interpretation of the
graded multiplicity for the spin coinvariant algebra as generic
degrees for (quantum) Hecke-Clifford algebras, etc. We hope to
return to these topics at another occasion.

\subsection{}
It is known \cite{Jo2, Se2, St2, Ya} (see Kleshchev
\cite[Chap.~13]{Kle}) that the representation theory of spin
symmetric group (super)algebra $\mathbb{C}S_n^-$ is
super-equivalent to its counterpart for Hecke-Clifford
(super)algebra $\HC := \mathcal{C}_n \rtimes \C S_n$; See
Section~\ref{subsec:super equiv} for notations and precise
formulations. (All the algebras and modules in this paper are
understood to admit a $\Z_2$-graded structure; however we will
avoid using the terminology of supermodules.) Let $D^\la_-$ denote
the simple $\mathbb{C}S_n^-$-module and $D^\la$ denote the simple
$\HC$-module, associated to a strict partition $\la$ of $n$. The
Clifford algebra $\Cl$ is naturally a simple module over the
algebra $\HC$ (which is identified with $D^{(n)}$), and it is the
counterpart of the basic spin $\mathbb{C}S_n^-$-module
$\mathcal{B}_n:=D^{(n)}_-$.

In Proposition~\ref{prop:mult.equiv} we show that,
for an arbitrary $S_n$-module $M$, the multiplicity problem for a simple
$\mathbb{C}S_n^-$-module $D^\la_-$ in $\mathcal B_n\otimes M$ is
essentially equivalent to the multiplicity problem for a simple
$\HC$-module $D^\la$ in $\Cl\otimes M$.
Therefore, in this paper, we shall mainly work with the algebra
$\HC$, keeping in mind that the results can be transferred to
the setting for $\spin$.

\subsection{}
Our first main result provides the graded multiplicity of the
simple $\HC$-module $D^\la$  in $\mathcal{C}_n
\otimes S^* V$ for $V=\C^n$. For a partition $\la$ of $n$ with
length $\ell(\la)$, we set
\begin{eqnarray*}
\delta(\lambda)= \left \{
 \begin{array}{ll}
 0,
 & \text{ if }\ell(\lambda) \text{ is even}, \\
 1
 , & \text{ if }\ell(\lambda) \text{ is odd}.
 \end{array}
 \right.
\end{eqnarray*}
If $\la$ is moreover a strict partition, we denote by $\la^*$ the
shifted diagram of $\la$, by $c_{ij}$  the content, and by
$h^*_{ij}$ the shifted hook length of the cell $(i,j) \in \la^*$
(see Section~\ref{sec:special} for precise definitions).

\begin{thma}   \label{th:S(V)}
Let $\la$ be a strict partition of $n$. The graded multiplicity of
$D^{\la}$ in the $\HC$-module $\mathcal{C}_n\otimes
S^* V$ is
\begin{align}   \label{hook}
2^{-\frac{\ell(\la)+\delta(\la)}{2}} \frac{t^{n(\la)}\prod_{(i,j)\in
\la^*}(1+t^{c_{ij}})} {\prod_{(i,j)\in \la^*}(1-t^{h^*_{ij}})}.
\end{align}
\end{thma}
The lowest degree term in \eqref{hook} is
$2^{\frac{\ell(\la)-\delta(\la)}{2}} t^{n(\la)}$, thanks to the
contribution $2^{\ell(\la)}$ from the product over the main diagonal
of $\la^*$ in the numerator. Theorem~A can be reformulated in
terms of the graded multiplicity of a coinvariant algebra which is
isomorphic to a graded regular module of $\HC$, see
Theorem~\ref{Cliff.tensor coin.}. In the spirit of a classical
theorem of Borel which identifies the coinvariant algebra of a
Weyl group with the cohomology ring of the corresponding flag
variety, the coinvariant algebra of $\HC$ should be regarded as the
cohomology ring (which has yet to be developed) of the flag
variety for the queer Lie supergroup.

To prove Theorem~A, we first obtain an expression of the graded
multiplicity in terms of the principal specialization of the Schur
$Q$-function, $Q_\la (t^\bullet) :=Q_\la (1,t,t^2,\ldots)$, and
then apply the following formula.

\begin{thmb}\label{th:t-Schur Q}
For a strict partition $\la$ of $n$, we have
\begin{align}
Q_\la (t^\bullet) =\frac{t^{n(\la)}\prod_{(i,j)\in
\la^*}(1+t^{c_{ij}})} {\prod_{(i,j)\in \la^*}(1-t^{h^*_{ij}})}.
\label{eqn:t-Schur Q}
\end{align}
\end{thmb}
It is well known (cf. \cite{Sa1, St}) that a Schur $Q$-function
can be written as a sum over the so-called marked shifted tableaux
of a given shape.
We establish in Theorem~\ref{thm:bijection} a bijection between marked shifted tableaux and
certain new combinatorial objects which we call  {\em colored
shifted tableaux}. Theorem~B is an easy consequence of such a bijection.

We show in Proposition~\ref{equiv} that the
formula \eqref{eqn:t-Schur Q} is equivalent to another formula of
Rosengren \cite[Proposition 3.1]{R}, who derived it from a Schur
function identity of Kawanaka~\cite{Ka}.

\subsection{}

Another result of this paper is a formula for the bi-graded
multiplicity of the simple $\HC$-module $D^\la$ in
$\mathcal{C}_n\otimes S^* V \otimes \wedge^* V$:
$$
\sum_{p=0}^\infty \sum_{q=0}^n t^p s^q m_\la(\Cl \otimes S^pV \otimes \wedge^qV).
$$
We shall adopt the following
short-hand notation for a specialization of Schur $Q$-function in
$2$-variables $s$ and $t$:
$$
Q_\la (t^\bullet; st^\bullet) :=Q_\la (1,t,t^2,\ldots;
s,st,st^2,\ldots).
$$

\begin{thmc}\label{th:Koszul}
Let $\la$ be a strict partition of $n$. The bi-graded multiplicity
of  $D^{\lambda}$ in the $\HC$-module
$\mathcal{C}_n\otimes S^{*}V\otimes \wedge^{*}V$ is
\begin{equation}  \label{eq:bi mult}
2^{-\frac{l(\lambda)+\delta(\lambda)}{2}}
Q_{\lambda}(t^{\bullet};st^{\bullet}).
\end{equation}
\end{thmc}
Setting $s=0$, we recover
Theorem~A from Theorem C. On the other hand, setting $t=0$, we
obtain a graded multiplicity formula of $D^\la$ in $\Cl \otimes
\wedge^*V$, see Corollary~\ref{mult wedge}.
We may also consider a Koszul $\Z$-grading
which counts the standard generators of $S^*V$ as
degree $2$ and the standard generators of $\wedge^*V$
as degree $1$. It follows by Theorem~C that, for the Koszul
grading, the graded multiplicity in $\mathcal{C}_n\otimes
S^*V\otimes \wedge^*V$ is given by the
same formula \eqref{hook} above. It will be nice to
obtain a closed formula for $Q_\la
(t^\bullet; st^\bullet)$.

\subsection{}

The paper is organized as follows. In Section~\ref{sec:special},
we provide a bijective proof of Theorem~B.  The graded multiplicities in $\mathcal{C}_n\otimes
S^*V$ are studied and Theorem~A is proved in
Section~\ref{sec:mult}. In Section~\ref{sec:bigraded}, we study
the bi-graded multiplicities in $\mathcal{C}_n\otimes
S^*V\otimes \wedge^*V$, and prove
Theorem~C.

{\bf Acknowledgments.} J.W. is supported by a semester dissertation fellowship
from the Department of Mathematics, University of Virginia.
W.W. is partially supported by NSF grant DMS-0800280.

\section{Principal specialization of Schur $Q$-functions}\label{sec:special}

In this section, we shall provide a bijective proof for Theorem~B,
after first recalling some basics on strict partitions and Schur
$Q$-functions (\cite{Sa1, St, Mac}).

\subsection{Strict partitions and shifted diagrams}

Let $n\in \Z_+$. We denote a composition
$\lambda=(\lambda_1,\lambda_2,\ldots)$ of $n$ by $\lambda\models
n$, and denote a partition $\la$ of $n$ by $\la \vdash n$. A
partition $\la$ will be identified with its Young diagram, that
is, $\la=\{(i,j)\in\mathbb{Z}^2~|~1\leq i\leq \ell(\lambda), 1\leq
j\leq \lambda_i\}$. To each cell $(i,j)\in \la$, we associate its content
$c_{ij}=j-i$ and
hook length $h_{ij}=\lambda_i+\lambda_j'-i-j+1$, where
$\lambda'=(\lambda_1',\lambda_2',\ldots)$ is the conjugate
partition of $\lambda$.

Suppose that the main diagonal of the Young diagram $\la$ contains
$r$ cells. Let $\alpha_i=\lambda_i-i$ be the number of cells in
the $i$th row of $\lambda$ strictly to the right of $(i,i)$, and
let $\beta_i=\lambda_i'-i$ be the number of cells in the $i$th
column of $\lambda$ strictly below $(i,i)$, for $1\leq i\leq r$.
We have $\alpha_1>\alpha_2>\cdots>\alpha_r\geq0$ and
$\beta_1>\beta_2>\cdots>\beta_r\geq0$. Then the Frobenius notation
for a partition is
$\lambda=(\alpha_1,\ldots,\alpha_r|\beta_1,\ldots,\beta_r)$. For
example, if $\lambda=(5,4,3,1)$, then $\alpha =(4,2,0),
\beta=(3,1,0)$ and hence $\lambda=(4,2,0|3,1,0)$ in Frobenius
notation.

Suppose that $\lambda$ is a strict partition of $n$, denoted by
$\lambda\vdash_s n$. Let $\la^*$ be the associated shifted
Young diagram, that is,
$$
\la^*=\{(i,j)~|~1\leq i\leq \ell(\lambda), i\leq j\leq\lambda_i+i-1
\}
$$
which is obtained from the ordinary Young diagram by shifting the
$k$th row to the right by $k-1$ squares, for each $k$.
Given $\lambda\vdash_s n$ with $\ell(\lambda)=\ell$, define its double
partition $\widetilde{\lambda}$ to be
$\widetilde{\lambda}=(\lambda_1,\ldots,\lambda_{\ell} |
\lambda_1-1,\lambda_2-1,\ldots,\lambda_{\ell}-1)$ in Frobenius
notation. Clearly, the shifted Young diagram $\la^*$ coincides
with the part of $\widetilde{\lambda}$ that lies above the main
diagonal. For each cell $(i,j)\in \lambda^*$, denote by $h^*_{ij}$
the associated hook length in the Young diagram
$\widetilde{\lambda}$, and set the content $c_{ij}=j-i$.

For example, let $\lambda= (4, 2, 1)$. The corresponding shifted
diagram and double diagram are
$$
\la^*=\young(\,\,\,\,,:\,\,,::\,)
\qquad \qquad
\widetilde{\lambda}=\young(\,\,\,\,\,,\,\,\,\,,\,\,\,\,,\,)
$$
The hook lengths and contents for each cell in $\la$
are respectively as follows:
$$
\young(6541,:32,::1)
\qquad \qquad \young(0123,:01,::0)
$$

\subsection{Schur $Q$-functions}

Let $\lambda$ be a strict partition with $\ell(\lambda)=\ell$. Suppose
$m\geq \ell$. The associated Schur $Q$-function
$Q_{\lambda}(z_1,z_2,\ldots,z_m)$ is defined by
\begin{align}
Q_{\lambda}(z_1,z_2,\ldots,z_m)=2^{\ell}\sum_{w\in
S_m/S_{m-\ell}}w\Big(z_1^{\lambda_1}\cdots z_{\ell}^{\lambda_{\ell}}
\prod_{1\leq i\leq \ell}\prod_{i<j\leq
m}\frac{z_i+z_j}{z_i-z_j}\Big),\label{SchurQ1}
\end{align}
where the symmetric group $S_m$ acts by permuting the variables
$z_1,\ldots,z_m$ and $S_{m-\ell}$ is the subgroup acting on
$z_{\ell+1},\ldots,z_{m}$. The definition of $Q_{\lambda}(z_1,z_2,\ldots,z_m)$
stabilizes as $m$ goes to infinity, and we write
$Q_{\lambda}(z) =Q_{\lambda}(z_1,z_2,\ldots)$, the symmetric functions
in infinitely many variables $z=(z_1,z_2,\ldots)$.
For $y=(y_1, y_2,\ldots)$, the following identity holds (see
\cite[III, \S 8]{Mac}):
\begin{align}
\prod_{i,j}\frac{1+y_iz_j}{1-y_iz_j} =\sum_{\lambda\text{:
strict}}2^{-\ell(\lambda)}Q_{\lambda}(y)Q_{\lambda}(z).
 \label{Cauchy identity}
\end{align}

It will be convenient to introduce another family of symmetric
functions $q_{\nu}(z)$ for any composition
$\nu=(\nu_1,\nu_2,\ldots)$ as follows:
\begin{align}
q_0(z)&=1\notag,\\
q_r(z)&=Q_{(r)}(z), \quad \text{ for }r\geq 1,\notag\\
q_{\nu}(z)&=q_{\nu_1}(z)q_{\nu_2}(z)\cdots.\notag
\end{align}
The generating function for $q_r(z)$ is
\begin{align}
\sum_{r\geq 0}q_r(z)u^r=\prod_i\frac{1+z_iu}{1-z_iu}.\label{gen.fun.qr}
\end{align}

We will write $q_r =q_r(z)$, etc.,
whenever there is no need to specify the variables.
Let $\Gamma_\C$ be the $\C$-algebra generated by $q_r, r\geq 1$, that is,
\begin{align}
\Gamma_\C =\C [q_1,q_2,\ldots].\label{Gamma}
\end{align}
Then $Q_{\lambda}$ for strict partitions $\la$ form a basis of
$\Gamma_\C$.

%
%
\subsection{Marked shifted tableaux and Schur $Q$-functions}

Denote by $\mathbf{P}'$ the ordered alphabet
$\{1'<1<2'<2<3'<3\cdots\}$. The symbols $1',2',3',\ldots$ are said
to be marked, and we shall denote by $|a|$ the unmarked version of
any $a\in\mathbf{P}'$; that is, $|k'| =|k| =k$ for each $k \in
\N$. For a strict partition $\la$, a {\it marked shifted tableau}
$T$ of shape $\lambda$, or a  {\it marked shifted $\la$-tableau}
$T$, is an assignment
$T:\la^*\rightarrow\mathbf{P}'$ satisfying:
\begin{itemize}
\item[(M1)] The letters are weakly increasing along each row and
column.

\item[(M2)] The letters $\{1,2,3,\ldots\}$ are strictly increasing
along each column.

\item[(M3)] The letters $\{1',2',3',\ldots\}$ are strictly
increasing along each row.   \label{M3}
\end{itemize}
%
%

For a marked shifted tableau $T$ of shape $\la$, let $\alpha_k$
be the number of cells $(i,j)\in \la^*$ such that $|T(i,j)|=k$ for
$k\geq 1$. The sequence $(\alpha_1,\alpha_2,\alpha_3,\ldots)$ is
called the {\em weight} of $T$. The Schur $Q$-function  associated
to $\lambda$ can be interpreted as (see \cite{Sa1, St, Mac})
\begin{align}
Q_{\lambda}(x)=\sum_{T}x^{T},\label{SchurQ}
\end{align}
where the summation is over all marked shifted $\la$-tableaux, and
$x^T=x_1^{\alpha_1}x_2^{\alpha_2}x_3^{\alpha_3}\cdots$ if $T$ has
weight $(\alpha_1,\alpha_2,\alpha_3,\ldots)$.
Denote by $|T|=\sum_{k\geq1}k\alpha_k$ if the weight of $T$ is $(\alpha_1,\alpha_2,\ldots)$.

\begin{example} \label{markedTableau0}
Suppose $\lambda=(5,4,2)$.
The following is an example of a marked shifted tableau of shape
$\lambda$ and its weight is $(2,5,4)$:
\begin{align}
T=\begin{tabular}{cccccc}
  \cline{1-5}
  \multicolumn{1}{|c|}{$1'$}&\multicolumn{1}{|c|}{$1$}
  &\multicolumn{1}{|c|}{$2'$}&\multicolumn{1}{|c|}{$2$}&\multicolumn{1}{|c|}{$2$}
  \\
  \cline{1-5}
  &\multicolumn{1}{|c|}{$2'$}&\multicolumn{1}{|c|}{$2$}&
  \multicolumn{1}{|c|}{$3'$}&\multicolumn{1}{|c|}{$3$}\\
  \cline{2-5}
  &&\multicolumn{1}{|c|}{$3'$}&\multicolumn{1}{|c|}{$3$}\\
  \cline{3-4}
\end{tabular}\notag
\end{align}
\end{example}

A {\em shifted reverse plane tableau} $S$
of shape $\lambda$ is a labeling of cells in the shifted diagram
$\lambda^*$ with nonnegative integers so that the rows and columns
are weakly increasing. Denote by $|S|$ the summation of the
entries in $S$. It is known (cf. \cite[Theorem 6.2.1]{Sa2})  that
\begin{align}
\sum_St^{|S|}=\prod_{(i,j)\in
\lambda}\frac{1}{1-t^{h^*_{ij}}},\label{Sagan}
\end{align}
summed over all shifted reverse plane
 tableaux of shape $\lambda$.

\subsection{A bijection theorem}

Let $\lambda$ be a strict partition. A {\it colored shifted
tableau} $C$ is an assignment $C:\la^*\rightarrow\mathbf{P}'$ such
that the associated assignment
$\overline{C}:\la^*\rightarrow\mathbb{Z}_+$ defined by
\begin{align}
\overline{C}(i,j)&=\left\{
\begin{array}{ll}
 |C(i,j)|-j, &\text{ if } C(i,j) \text{ is marked}, \\
     |C(i,j)|-i, &\text{ if } C(i,j) \text{ is unmarked}
     \end{array}
 \right.\notag
\end{align}
is a shifted reverse plane tableau of shape $\lambda$. The weight of a colored
shifted tableau is defined in the same way as for marked shifted
tableaux.  Denote by $|C|=\sum_{k\geq1}k\alpha_k$ if the weight of $C$ is $(\alpha_1,\alpha_2,\ldots)$.

\begin{thm}\label{thm:bijection}
Suppose that $\lambda$ is a strict partition of $n$ and
$\alpha=(\alpha_1,\alpha_2,\ldots)$ is a composition of $n$. Then
there exists a bijection between the set of marked shifted
$\la$-tableaux of weight $\alpha$ and the set of
colored shifted $\la$-tableaux of weight $\alpha$.
\end{thm}

\begin{proof}
Suppose that $T$ is a marked shifted tableau of shape $\lambda$ and
weight $\alpha=(\alpha_1,\alpha_2,\ldots)$.
Set $m=\max\{|T(i,j)|~|~(i,j)\in \la^*\}$. For each $1 \le k \le m$,
$\la^{k,*} =\{(i,j)\in\la^*~|~|T(i,j)|\leq k\}$ is a shifted diagram of a strict partition $\la^k$,
and $\la^{1} \subseteq \la^{2}\subseteq\cdots\subseteq  \la^{m}$.

We shall construct by induction on $k$ a chain of colored shifted tableaux $T^k$ of shape
$\la^{k}$ and weight $\alpha^k=(\alpha_1,\ldots,\alpha_k)$, for $1\leq k\leq m$.
Set
$T^1: \la^{1,*}\rightarrow \mathbf{P}'$ to be the restriction of $T$ to $\la^{1,*}$.
Since $T$ is a marked shifted tableau, $\la^{1}$ is a one-row partition
and hence $T^1$ is already a colored shifted tableau of weight $\alpha^1=(\alpha_1)$.

Suppose that $T^{k-1}$ is a colored shifted tableau of shape
$\la^{k-1}$ and weight $\alpha^{k-1}=(\alpha_1,\ldots,\alpha_{k-1})$.
In order to construct $T^k$ from $T^{k-1}$,
we start with an intermediate tableau $T_k$ defined by
\begin{align*}
T_k: \la^{k,*} &\longrightarrow \mathbf{P}'\notag\\
       (i,j)&\mapsto\left\{\begin{array}{ll}
                           T^{k-1}(i,j), & \text{ if }(i,j)\in \la^{(k-1),*} \\
                           T(i,j),& \text{ if }(i,j)\in \la^{k,*} / \la^{(k-1),*}.
                         \end{array}\right.
\end{align*}
There is at most one cell labeled by
$k'$ in each row of $T_{k}$ since $T$ satisfies~(M3).
Suppose that the cells labeled by $k'$ in
$T_{k}$ are $(i_1,j_1),(i_2,j_2),\ldots,(i_p,j_p)$ with
$i_1<i_2<\cdots <i_p$. Start with the topmost cell $(i_1,j_1)$ and label
its left and upper neighboring cells $(i_1,j_1-1)$ and
$(i_1-1,j_1)$, if they exists,  as
\begin{center}
\begin{tabular}{cc}
  \cline{2-2}&\multicolumn{1}{|c|}{$c$}\\
  \cline{1-2}
  \multicolumn{1}{|c|}{$b$}&\multicolumn{1}{|c|}{$k'$}\\
  \cline{1-2}
\end{tabular}
\end{center}
(In case when either the left or the upper neighboring cell is missing, the
exchange procedure below is simplified in an obvious manner).
Set
\begin{align*}
\bar{b}=\left\{
\begin{array}{ll}
 |b|-(j_1-1), &\text{ if } b \text{ is marked}, \\
     |b|-i_1, &\text{ otherwise;}
     \end{array}
      \right.
 \end{align*}
 \begin{align*}
 \bar{c}=\left\{
\begin{array}{ll}
 |c|-j_1, &\text{ if }c \text{ is marked}, \\
     |c|-(i_1-1), &\text{ otherwise.}
     \end{array}
 \right.
\end{align*}

If $k-j_1<\bar{b}$ or
$k-j_1<\bar{c}$, exchange $k'$ and $b$ or $c$ in $T_k$ as follows:
\begin{align}
\text{Case I}\; (\bar{b}\geq\bar{c}):\qquad & \begin{tabular}{ccc}
  \cline{2-3}&\multicolumn{2}{|c|}{$c$}\\
  \cline{1-3}
  \multicolumn{1}{|c|}{$b$}&\multicolumn{2}{|c|}{$k'$}\\
  \cline{1-3}
\end{tabular}
   \longrightarrow
\begin{tabular}{ccc}
  \cline{3-3}&&\multicolumn{1}{|c|}{$c$}\\
  \cline{1-3}
  \multicolumn{2}{|c|}{$k'$}&\multicolumn{1}{|c|}{$b$}\\
  \cline{1-3}
\end{tabular}   \notag\\
 \text{ Case II }\;
(\bar{c}>\bar{b}):\qquad &\begin{tabular}{ccc}
  \cline{2-3}&\multicolumn{2}{|c|}{$c$}\\
  \cline{1-3}
  \multicolumn{1}{|c|}{$b$}&\multicolumn{2}{|c|}{$k'$}\\
  \cline{1-3}
\end{tabular}\longrightarrow\begin{tabular}{ccc}
  \cline{2-3}&\multicolumn{2}{|c|}{$k'$}\\
  \cline{1-3}
  \multicolumn{1}{|c|}{$b$}&\multicolumn{2}{|c|}{$c$}\\
  \cline{1-3}
\end{tabular}    \notag
\end{align}
Note that $b$ is unmarked  in Case~I and
$c$ is unmarked in Case~II.
Hence the resulting diagram
\begin{align*}
\text{I: } \quad \begin{tabular}{ccc}
  \cline{2-3}&\multicolumn{2}{|c|}{$c$}\\
  \cline{1-3}
  \multicolumn{1}{|c|}{$k'$}&\multicolumn{2}{|c|}{$b$}\\
  \cline{1-3}
\end{tabular}
\qquad \text{ or } \qquad
\text{II: } \quad \begin{tabular}{ccc}
  \cline{2-3}&\multicolumn{2}{|c|}{$k'$}\\
  \cline{1-3}
  \multicolumn{1}{|c|}{$b$}&\multicolumn{2}{|c|}{$c$}\\
  \cline{1-3}
\end{tabular}
\end{align*}
satisfies the requirement for colored shifted tableaux.
Keep repeating the above procedure for the new cells occupied with this $k'$, until it stops.
Then move on to apply
the same procedure above to the cells $(i_2,j_2),  \ldots, (i_p,j_p)$
one by one, and denote by $T^k$ the resulting tableau in the end.

We claim that $T^k$ is a colored shifted tableau. By induction
hypothesis, $T^{k-1}$ is a colored shifted tableau. Clearly, the
exchange procedure above by definition ensures that the
requirement being a colored shifted tableau is already fulfilled
for the cells in $T^k$ other than those occupied by $k$. So it
remains to check the conditions on each cell $(i,j)\in \la^{k,*}$
with $T^k(i,j)=k$. Assume that the cell $(i,j-1)$ in $T^k$, if
it exists, is labeled by $d\in\mathbf{P}'$. Note that $j-1\geq i$
and $|d|\leq k$. If $d$ is unmarked, then $|d|-i\leq k-i$. If $d$
is marked, then $|d|-(j-1)\leq|d|-i\leq k-i$. Similarly, assume
that the cell $(i-1,j)$ in $T^k$, if it exists, is labeled by
$e\in\mathbf{P}'$. For unmarked $e$, we have $|e|<k$ or
equivalently $|e|-(i-1)\leq k-i$, since there is at most one
unmarked $k$ in each column of $T^k$. For marked $e$, we have
$|e|\leq k$ and hence $|e|-j \leq k-i$, since $j\geq i$. This
proves the claim.

Hence, we have constructed a colored shifted tableau $T^m$
of the same shape and weight as for $T$ which we started with.


We claim the exchange procedure above from $T$ to $T^m$ is reversible. 
It suffices to show that the above procedure from $T^{k-1}$
to $T^k$ is invertible for each $k$. Denote by $T^{k,0}$ the resulting tableau
after removing cells labeled by unmarked $k$ from $T^k$.
There exists at most one cell labeled by marked $k'$ in each row of $T^{k,0}$,
and suppose that these cells are
$(i_1,j_1),(i_2,j_2),\ldots,(i_p,j_p)$ in $T^{k,0}$ with $i_1>i_2>\ldots>i_p$.
Start with the lowest cell $(i_1,j_1)$ and suppose that
its right and lower neighboring cells $(i_1,j_1+1)$ and $(i_1+1,j_1)$, if they exist, in $T^{k,0}$
are labeled by $b,c\in\mathbf{P}'$ as follows:
\begin{align}
\begin{tabular}{ccc}
  \cline{1-3}\multicolumn{2}{|c|}{$k'$}&\multicolumn{1}{|c|}{$b$}\\
  \cline{1-3}
  \multicolumn{2}{|c|}{$c$}&\\
  \cline{1-2}
\end{tabular}\notag
\end{align}
Set
\begin{align}
\tilde{b}=\left\{
\begin{array}{ll}
 |b|-(j_1+1), &\text{ if } b \text{ is marked}, \\
     |b|-i_1, &\text{ otherwise},
     \end{array}\notag
 \right.
\end{align}

\begin{align}
 \tilde{c}=\left\{
\begin{array}{ll}
 |c|-j_1, &\text{ if }c \text{ is marked}, \\
     |c|-(i_1+1), &\text{ otherwise}.
     \end{array}\notag
 \right.
\end{align}
If $k'>b$ or $k'>c$, exchange $k'$  and $b$ or $c$ in $T^{k,0}$ as follows:
\begin{align}
\text{Case I} \; (\tilde{b}\leq \tilde{c}):\qquad &\begin{tabular}{ccc}
  \cline{1-3}\multicolumn{2}{|c|}{$k'$}&\multicolumn{1}{|c|}{$b$}\\
  \cline{1-3}
  \multicolumn{2}{|c|}{$c$}&\\
  \cline{1-2}
\end{tabular}
\longrightarrow
\begin{tabular}{ccc}
  \cline{1-3}\multicolumn{1}{|c|}{$b$}&\multicolumn{2}{|c|}{$k'$}\\
  \cline{1-3}
  \multicolumn{1}{|c|}{$c$}&&\\
  \cline{1-1}
\end{tabular}\notag\\
\text{Case II} \; (\tilde{b}>\tilde{c}):\qquad  &\begin{tabular}{ccc}
  \cline{1-3}\multicolumn{2}{|c|}{$k'$}&\multicolumn{1}{|c|}{$b$}\\
  \cline{1-3}
  \multicolumn{2}{|c|}{$c$}&\\
  \cline{1-2}
\end{tabular}
\longrightarrow
\begin{tabular}{ccc}
  \cline{1-3}\multicolumn{2}{|c|}{$c$}&\multicolumn{1}{|c|}{$b$}\\
  \cline{1-3}
  \multicolumn{2}{|c|}{$k'$}&\\
  \cline{1-2}
\end{tabular}\notag
\end{align}
Keep repeating the above procedure to the new cell occupied by this $k'$,
until it stops. Then move on to the cells $(i_2,j_2), \ldots, (i_p,j_p)$
one by one and apply
the same procedure. Denote by $T^{k,1}$ the resulting tableau in the end.
Removing the cells labeled by $k'$ from $T^{k,1}$, we recover the tableau $T^{k-1}$.
\end{proof}

\begin{example} Suppose $\lambda=(5,4,2)$ and
$T$ is the marked shifted tableau given by Example~\ref{markedTableau0}.
Then the colored shifted tableau corresponding to $T$ is
$T^3$, where
\begin{align*}
T^1&=\begin{tabular}{cc}
  \cline{1-2}
  \multicolumn{1}{|c|}{$1'$}&\multicolumn{1}{|c|}{$1$}
  \\
  \cline{1-2}
\end{tabular}
\qquad\quad T^2=\begin{tabular}{ccccc}
  \cline{1-5}
  \multicolumn{1}{|c|}{$1'$}&\multicolumn{1}{|c|}{$2'$}
  &\multicolumn{1}{|c|}{1}&\multicolumn{1}{|c|}{2}&\multicolumn{1}{|c|}{$2$}\\
  \cline{1-5}
  &\multicolumn{1}{|c|}{$2'$}&\multicolumn{1}{|c|}{$2$}\\
  \cline{2-3}
\end{tabular}\notag\\ \\
 & T^3 =\begin{tabular}{ccccc}
  \cline{1-5}
  \multicolumn{1}{|c|}{$1'$}&\multicolumn{1}{|c|}{$2'$}
  &\multicolumn{1}{|c|}{$3'$}&\multicolumn{1}{|c|}{$1$}&\multicolumn{1}{|c|}{2}\\
  \cline{1-5}
  &\multicolumn{1}{|c|}{$2'$}&\multicolumn{1}{|c|}{$2$}
  &\multicolumn{1}{|c|}{$2$}&\multicolumn{1}{|c|}{$3$}\\
  \cline{2-5}
  &&\multicolumn{1}{|c|}{$3'$}&\multicolumn{1}{|c|}{$3$}\\
  \cline{3-4}
\end{tabular}\notag
\end{align*}
\end{example}

\subsection{Proof of Theorem~B}

\begin{proof}
It follows by Theorem~\ref{thm:bijection} that
\begin{align}  \label{T=C}
\sum_T t^{|T|} =\sum_{C}t^{|C|},
\end{align}
where the first summation is over all marked shifted $\la$-tableaux $T$
and  the second summation is over all colored shifted $\la$-tableaux $C$.
The left hand side of \eqref{T=C} is equal to $Q_{\lambda}(t,t^2,t^3, \ldots)$ by (\ref{SchurQ}).

It follows from the definition of colored shifted tableaux and then \eqref{Sagan}  that
\begin{align*}
 \sum_{C}t^{|C|}
 &=\Big(\prod_{(i,j)\in
 \la^*}(t^i+t^j)\Big)\sum_{S}t^{|S|} \\
 &=\frac{\prod_{(i,j)\in \la^*}(t^i+t^{j})}
{\prod_{(i,j)\in \la^*}(1-t^{h^*_{ij}})},
\end{align*}
where the summation on $S$ is taken over all shifted reverse plane tableaux of shape $\lambda$.

Putting everything together, we obtain that
\begin{align}
 Q_{\lambda}(t^{\bullet}) &=\frac{1}{t^n}Q_{\lambda}(t,t^2,t^3,\ldots)
 =\frac{1}{t^n}\sum_T t^{|T|}  \\
 &=
 \frac{1}{t^n}\frac{\prod_{(i,j)\in \la^*}(t^i+t^{j})}
{\prod_{(i,j)\in \la^*}(1-t^{h^*_{ij}})}
=\frac{t^{n(\lambda)}\prod_{(i,j)\in \la^*}(1+t^{c_{ij}})}
{\prod_{(i,j)\in \la^*}(1-t^{h^*_{ij}})}.\notag
\end{align}
This completes the proof of Theorem~B.
\end{proof}

\begin{rem}
It follows from the proof above that Theorem B can be restated as
\begin{align*}
Q_\la (t^\bullet)
=\frac{\prod_{(i,j)\in\la^*}(t^{i-1} +t^{j-1})} {\prod_{(i,j)\in \la^*}(1-t^{h^*_{ij}})}.
\end{align*}
\end{rem}

\subsection{Another formula for $Q_{\lambda}(t^{\bullet})$}

For $k\in\N$, we set
$$
(a;t)_k=(1-a)(1-at)\cdots(1-at^{k-1}).
$$
Rosengren \cite[Proposition~3.1]{R} has obtained the following
formula for $Q_{\lambda}(t^{\bullet})$, starting from a Schur
function identity of Kawanaka:
\begin{align}
 Q_{\lambda}(t^{\bullet})
 =\prod_{1\leq i\leq \ell(\la)}\frac{(-1;t)_{\la_i}}{(t;t)_{\la_i}}
 \prod_{1\leq i<j\leq \ell(\la)}
 \frac{t^{\la_j}-t^{\la_i}}{1-t^{\la_i+\la_j}}.
 \label{eqn(R):t-Schur Q}
\end{align}
\begin{prop}\label{equiv}
The formula~(\ref{eqn:t-Schur Q}) is equivalent to~Rosengren's
formula (\ref{eqn(R):t-Schur Q}).
\end{prop}

\begin{proof}
Set $\ell =\ell(\la)$. It is known (cf. \cite[III, \S 8, Ex.~12]{Mac}) that in
the $i$th row of $\la^*$, the hook lengths $h^*_{ij}$ for $i\leq
j\leq\la_i+i-1$ are
$1,2,\ldots,\la_i,\la_i+\la_{i+1},\la_i+\la_{i+2},\ldots,\la_{i}+\la_{\ell}$
with exception
$\la_i-\la_{i+1},\la_i-\la_{i+2},\ldots,\la_{i}-\la_{\ell}$. Hence we
have
\begin{align}
\prod_{(i,j)\in\la^*}\frac{1}{1-t^{h^*_{ij}}}
=\frac{1}{\prod_{1\leq i\leq \ell}(t;t)_{\la_i}}\prod_{1\leq i<j\leq
\ell} \frac{1-t^{\la_i-\la_j}}{1-t^{\la_i+\la_j}}.
 \label{shifted hook}
\end{align}
The equivalence between~(\ref{eqn:t-Schur Q})
and~(\ref{eqn(R):t-Schur Q}) can now be deduced by applying
\eqref{shifted hook} and noting that the contents $c_{ij}$ for
$i\le j \le \la_i +i-1$ are $0,1,\ldots, \la_i-1$.
\end{proof}

\begin{rem}
By~(\ref{SchurQ}), $Q_{\la}(1^m)$ is
equal to the number of marked shifted Young tableaux of shape
$\lambda$ with fillings by letters $\le m$. On the other hand, it
follows from~\cite[Theorem~4]{Se1} that
$2^{\frac{\delta(\la)-\ell(\la)}{2}}Q_{\la}(1^m)$ gives the dimension
of the irreducible representation of the queer Lie superalgebra
$\mathfrak{q}(m)$ of highest weight $\la$.
\end{rem}

\section{The graded multiplicity in $\mathcal{C}_n\otimes
S^{*}V$} \label{sec:mult}

The goal of this section is to establish Theorem~A. In addition, a
tensor identity in Lemma~\ref{lem:equiv.tensor} allows us to
translate a multiplicity problem for $\HC$ to $\spin$, and vice versa
(see Proposition~\ref{prop:mult.equiv}).

\subsection{Some basics about superalgebras}

We shall recall some basic notions of superalgebras, referring the
reader to~\cite[Chapter 12]{Kle}. Let us denote by
$\bar{v}\in\mathbb{Z}_2$ the parity of a homogeneous vector $v$ of a
vector superspace. A superalgebra $\mathcal{A}$ is a
$\mathbb{Z}_2$-graded associative algebra.
An $\mathcal{A}$-module always means a $\mathbb{Z}_2$-graded
left $\mathcal{A}$-module in this paper. A homomorphism $f:V\rightarrow W$ of
$\mathcal{A}$-modules $V$ and $W$ means a linear map such that $
f(av)=(-1)^{\bar{f}\bar{a}}af(v).$  Note that this and other such
expressions only make sense for homogeneous $a, f$ and the meaning
for arbitrary elements is attained by extending linearly from
the homogeneous case. Let $V$ be a finite dimensional
$\mathcal{A}$-module. Let $\Pi
 V$ be the same underlying vector space but with the opposite
 $\mathbb{Z}_2$-grading. The new action of $a\in\mathcal{A}$ on $v\in\Pi
 V$ is defined in terms of the old action by $a\cdot
 v:=(-1)^{\bar{a}}av$.
Denote by $\mathcal{A}\text{-smod}$ the category of finite dimensional $\mathcal{A}$-modules.


Given two superalgebras $\mathcal{A}$ and $\mathcal{B}$,
the tensor product  $\mathcal{A}\otimes\mathcal{B}$
is naturally a superalgebra.
Suppose that $V$ is an $\mathcal{A}$-module and $W$ is a
$\mathcal{B}$-module. Then the tensor space $V\otimes W$ affords an $\mathcal{A}\otimes \mathcal{B}$-module,
denoted by $V\boxtimes W$, via
$$
(a\otimes b)(v\otimes w)=(-1)^{\bar{b}\bar{v}}av\otimes bw,\quad a\in \mathcal{A},
b\in \mathcal{B}, v\in V, w\in W.
$$

\subsection{Spin symmetric group algebras $\mathbb{C}S_n^-$
and Hecke-Clifford algebras $\HC$}\label{subsec:super equiv}

Recall that the spin symmetric group algebra $\C S_n^-$ is the algebra generated by
 $t_1,t_2,\ldots,t_{n-1}$ subject to the relations:
\begin{align}
t_i^2&=1,\quad 1\leq i\leq n-1\notag\\
t_it_{i+1}t_i&=t_{i+1}t_it_{i+1},\quad 1\leq i\leq n-2\notag\\
t_it_j&=-t_jt_i,\quad 1\leq i,j\leq n-1,|i-j|\geq 1.\notag
\end{align}
$\mathbb{C}S_n^-$ is a superalgebra with each $t_i$ being odd, for $1\leq i\leq n-1$.

Denote by $\mathcal{C}_n$
the Clifford superalgebra generated by the odd elements $c_1,\ldots,c_n$,
subject to the relations $c_i^2=1,c_ic_j=-c_jc_i$ for $1\leq i\neq
j\leq n$. Observe that $\mathcal{C}_n$ is a simple superalgebra
and there is a unique (up to isomorphism) irreducible
$\Cl$-module $U_n$.

Define the Hecke-Clifford
algebra $\HC=\mathcal{C}_n\rtimes\C S_n$ to be the
superalgebra generated by odd elements $c_1,\ldots,c_n$ and even
elements $s_1,\ldots,s_{n-1}$, subject to the relations:
\begin{align*}
s_i^2&=1, s_is_j=s_js_i,\quad 1\leq i,j\leq n-1, |i-j|>1,\\
s_is_{i+1}s_i&=s_{i+1}s_is_{i+1},\quad 1\leq i\leq n-2,\\
c_i^2&=1,c_ic_j=-c_jc_i,\quad 1\leq i\neq
j\leq n,\\
s_ic_i&=c_{i+1}s_i, s_ic_j=c_js_i,\quad 1\leq i,j\leq n-1, j\neq i,i+1.
\end{align*}
There is a superalgebra isomorphism  (cf.
\cite{Se1, Ya}):
\begin{align}
 \label{map:isorm.HC}
\begin{split}
\C S_n^-\otimes\mathcal{C}_n&\longrightarrow\HC \\
c_i&\mapsto c_i, \quad 1\leq i\leq n,    \\
t_j&\mapsto \frac{1}{\sqrt{-2}}s_j(c_j-c_{j+1}),\quad 1\leq j\leq n-1.
\end{split}
\end{align}
The two exact functors
\begin{eqnarray*}
\mathfrak{F}_n :=-\boxtimes U_n:
& \C
S_n^-\text{-smod} \rightarrow\HC\text{-smod},\notag\\
\mathfrak{G}_n :={\rm Hom}_{\mathcal{C}_n}(U_n,-):
& \HC
\text{-smod} \rightarrow\C S_n^-\text{-smod}  \notag
\end{eqnarray*}
define Morita super-equivalence between the superalgebras $\HC$ and $\spin$
(cf. Kleshchev \cite[Proposition~13.2.2]{Kle} for precise details).

It is known \cite{Jo2, Se1, St2} (cf. \cite{Kle}) that for each
strict partition $\la$ of $n$, there exists an irreducible
$\HC$-module $D^{\la}$ and $\{D^{\la}~|~\la\vdash_s n\}$
forms a complete set of non-isomorphic irreducible
$\HC$-modules.
We have a complete set of
non-isomorphic irreducible $\spin$-modules
$\{D^\la_-~|~\la\vdash_sn\}$,
and by \cite[Proposition~13.2.2]{Kle},
\begin{align}
\mathfrak{G}_n(D^{\la}) = \left\{\begin{array}{ll}
 D^\la_-,
 & \text{ if }n\text{ or } \ell(\la)\text{ is even},\\
 D^\la_-\bigoplus\Pi D^\la_-,
& \text{ otherwise}.
\end{array}
\right. \label{corresp:spin HC}
\end{align}

Denote the trivial representation  by ${\bf 1}$
and the sign representation of $S_n$ by   ${\rm sgn}$. Note that
$\mathcal{C}_n\cong{\rm ind}^{\HC}_{\mathbb{C}S_n}{\bf
1}$ is the irreducible $\HC$-module $D^{(n)}$
\cite[Lemma~22.2.4]{Kle}. It follows from
~(\ref{corresp:spin HC})~
that the  irreducible
$\mathbb{C}S_n^-$-module $\mathcal{B}_n:=D^{(n)}_-$ satisfies that
\begin{align}
{\rm Hom}_{\mathcal{C}_n}(U_n, \mathcal{C}_n)\cong\left\{
\begin{array}{ll}
\mathcal{B}_n,\quad &\text{ if } n \text{ is even},\\
\mathcal{B}_n\bigoplus\Pi\mathcal{B}_n, \quad &\text{ if } n \text{ is odd}.
\end{array}
\right.\label{basic spin}
\end{align}
It can be shown that the $\mathbb{C}S_n^-$-module
$\mathcal{B}_n$ coincides with the basic spin representation $L_n$
defined in~\cite[2C]{Jo1}.

%
%
\subsection{The multiplicity problem of $\HC$ vs $\spin$}

Given a $\C S_n$-module $M$ and a $\C S_n^-$-module $E$,
the tensor product $E\otimes M$ affords a $\C S_n^-$-module as follows:
\begin{align}
t_j(u\otimes x)=(t_ju)\otimes (s_jx),
\quad 1\leq j\leq n-1, u\in E, x\in M.\label{eqn:spin.tensor sym}
\end{align}
Meanwhile, the tensor product $F\otimes M$ of a $\HC$-module $F$
and a $\C S_n$-module $M$ naturally affords a $\HC$-module
with
\begin{align}
c_i(u\otimes x)=(c_iu)\otimes x,\qquad s_j(u\otimes x)=(s_ju)\otimes (s_jx),
\label{eqn:HC.tensor sym}
\end{align}
for $1\leq i\leq n, 1\leq j\leq n-1, u\in F, x\in M.$

\begin{lem} [A tensor identity]
 \label{lem:equiv.tensor}
Suppose that $M$ is a $\C S_n$-module.
Then we have an isomorphism of $\C S_n^-$-modules:
$\mathfrak G_n (\Cl) \otimes M \cong \mathfrak G_n (\Cl \otimes M)$; that is,
\begin{align}
{\rm Hom}_{\mathcal{C}_n}(U_n,\mathcal{C}_n)\otimes M\cong{\rm
Hom}_{\mathcal{C}_n}(U_n, \mathcal{C}_n\otimes M) .\notag
\end{align}
\end{lem}
\begin{proof}
Observe that by~(\ref{map:isorm.HC}) the action of $\spin$ on
${\rm Hom}_{\mathcal{C}_n}(U_n, \mathcal{C}_n\otimes M)$ is given
by
\begin{align}
(t_j*f)(u)=(\frac{1}{\sqrt{-2}}s_j(c_j-c_{j+1}))(f(u)), \quad
f\in{\rm Hom}_{\mathcal{C}_n}(U_n, \mathcal{C}_n\otimes M), u\in
U_n\label{eqn:action1}
\end{align}
while by~(\ref{eqn:spin.tensor sym}) the $\spin$-module structure
of ${\rm Hom}_{\mathcal{C}_n}(U_n,\mathcal{C}_n)\otimes M$ is given by
\begin{align}
t_j*(f\otimes x)=(t_j*f)\otimes(s_j x),\quad f\in{\rm
Hom}_{\mathcal{C}_n}(U_n,\mathcal{C}_n), x\in
M.\label{eqn:action2}
\end{align}
Define a map
\begin{align}
\phi:{\rm Hom}_{\mathcal{C}_n}(U_n,\mathcal{C}_n)\otimes M&\longrightarrow
{\rm Hom}_{\mathcal{C}_n}(U_n, \mathcal{C}_n\otimes M),\notag\\
f\otimes x&\mapsto (u\mapsto f(u)\otimes x).\notag
\end{align}
Clearly $\phi$ is injective and thus an isomorphism of vector spaces by a dimension
counting argument. It remains to show that $\phi$ is a
$\spin$-module homomorphism. Indeed, for  $u\in U_n$, $f\in{\rm
Hom}_{\mathcal{C}_n}(U_n,\mathcal{C}_n)$ and $ x\in M$, we have
\begin{align*}
\phi(t_j*(f\otimes x))(u)
&=\phi(t_j*f\otimes s_jx)(u) \quad  \text{ by}~~(\ref{eqn:action2})\notag\\
&=(t_j*f)(u)\otimes s_jx\notag\\
&=((\frac{1}{\sqrt{-2}}s_j(c_j-c_{j+1}))f(u))\otimes s_jx\notag\\
&=(\frac{1}{\sqrt{-2}}s_j(c_j-c_{j+1}))(f(u)\otimes x)
 \quad \text{ by~}(\ref{eqn:HC.tensor sym})\notag\\
&=(t_j*\phi(f\otimes x))(u)  \quad \text{ by}~(\ref{eqn:action1}).
\end{align*}
\end{proof}

\begin{prop}\label{prop:mult.equiv}
Suppose that $M$ is a $\C S_n$-module. Let $m_{\la}$ and
$m_{\la}^-$ be the multiplicities of $D^{\la}$ and
$D^\la_-$ in the $\HC$-module $\Cl\otimes M$ and
$\spin$-module $\mathcal{B}_n\otimes M$, respectively. Then,
\begin{align*}
m_{\la}^- =\left\{\begin{array}{ll}
m_{\la}, & \text{ if }n \text{ is even},\\
m_{\la}, & \text{ if }n \text{ is odd} \text{ and } \ell(\la) \text{ is odd},\\
\frac12 m_{\la}, & \text{ if }n \text{ is odd} \text{ and } \ell(\la)
\text{ is even}.
\end{array}
\right.
\end{align*}
\end{prop}
\begin{proof}
It follows by definition that
\begin{align}
\mathcal{C}_n\otimes M\cong\bigoplus_{\la\vdash_s n}m_{\la} D^{\la},
 \qquad
\mathcal{B}_n\otimes M\cong\bigoplus_{\la\vdash_s n}m_{\la}^-
D^\la_-\label{eqn:Bl M}.
\end{align}
By~(\ref{basic spin}) and Lemma~\ref{lem:equiv.tensor}, we have
\begin{align*}
\mathfrak{G}_n(\mathcal{C}_n\otimes M)\cong\left\{
\begin{array}{ll}
\mathcal{B}_n\otimes M,\quad &\text{ if } n \text{ is even},\\
2\mathcal{B}_n\otimes M,\quad &\text{ if } n \text{ is odd}.
\end{array}
\right.
\end{align*}
This together with (\ref{eqn:Bl M}) implies
that
\begin{align*}
\bigoplus_{\la\vdash_s n}m_{\la} \mathfrak{G}_n(D^{\la})\cong
\left\{\begin{array}{ll}
\bigoplus_{\la\vdash_s n}m_{\la}^- D^\la_-, & \text{ if }n \text{ is even}, \\
\bigoplus_{\la\vdash_s n}2m_{\la}^- D^\la_-, & \text{ if }n
\text{ is odd}.
\end{array}
\right.
\end{align*}
The proposition now follows by comparing the multiplicities of
$D^\la_-$ on both sides  and using~(\ref{corresp:spin HC}).
\end{proof}

\subsection{Proof of Theorem~A}

The symmetric group $S_n$ acts naturally on the ($\Z_+$-graded)
symmetric algebra  on $V =\C^n$:
$$
S^*V =\bigoplus_{j \ge 0} S^jV.
$$
As $S_n$-modules, we will identify $ S^*V$ with the algebra
of polynomials in $n$ variables over $\C$. Note that
$\mathcal{C}_n\otimes S^* V={\rm ind}^{\HC}_{\C
S_n} S^* V$ is naturally a $\Z_+$-graded
$\HC$-module, with the grading inherited from the one on
$ S^* V$.

\begin{lem}  \label{lem:induced}
We have the following isomorphism of
$\HC$-modules for $j \ge 0$:
\begin{align}
\mathcal{C}_n\otimes S^jV\cong \bigoplus_{\nu\models
n,n(\nu)=j} {\rm ind}^{\HC}_{\C S_{\nu}} {\bf 1}. \notag
\end{align}
\end{lem}

\begin{proof}
Identify $S^*V \equiv \C[x_1,\ldots, x_n]$.
The definition \eqref{n lambda} of  $n(\nu)$ makes sense for any composition $\nu$.
The representatives of
the $S_n$-orbits
on the set of all monomials of degree $j$ in $\C[x_1,\ldots, x_n]$
can be chosen to be
$$
(x_1\ldots x_{\nu_1})^0 (x_{\nu_1+1} \ldots x_{\nu_1+\nu_2})^1
 (x_{\nu_1+\nu_2+1} \ldots x_{\nu_1+\nu_2+\nu_3})^2 \ldots
$$
where $\nu =(\nu_1, \nu_2, \ldots)$ runs over all compositions of $n$ such that $n(\nu)=j$.
Then, as $\mathbb{C}S_n$-modules,
$$
S^jV\cong\bigoplus_{\nu\models n,n(\nu)=j} {\rm ind}^{\C
S_n}_{\C S_{\nu}} {\bf 1}.
$$
Hence, as $\HC$-modules, we have
\begin{align*}
\mathcal{C}_n\otimes S^jV
\cong \bigoplus_{\nu\models n,n(\nu)=j}
{\rm ind}^{\HC}_{\C S_n}  {\rm ind}^{\C S_n}_{\C S_{\nu}} {\bf 1}
\cong \bigoplus_{\nu\models
n,n(\nu)=j} {\rm ind}^{\HC}_{\C S_{\nu}} {\bf 1}.
\end{align*}
\end{proof}

Below, we shall denote by $[u^n] f(u)$ the coefficient of $u^n$
of a formal power series $f(u)$ in a variable $u$. We are
ready to prove Theorem~A in Introduction.

\begin{proof}[Proof of Theorem~A]
Denote by $K(\HC\text{-smod})$ the complexified Grothendieck group of the
category $\HC\text{-smod}$. Recall the definition of the
algebra $\Gamma_\C$ from~(\ref{Gamma}).
There exists
an isomorphism called the characteristic map~\cite{Se1, St2, Jo2}
\begin{align*}
{\rm ch}: \bigoplus_{n\geq
0}K(\HC\text{-smod})&\longrightarrow\Gamma_\C,
\end{align*}
which sends $D^{\la}$ to
$2^{-\frac{\ell(\la)-\delta(\la)}{2}}Q_{\la}$ for all strict
partitions $\la$. It is known that
\begin{align}
{\rm ch} \big ( {\rm ind}^{\HC}_{\mathbb{C}S_{\nu}}{\bf
1}\big ) = q_{\nu},  \quad \forall \nu \models n.   \label{char.
map}
\end{align}

By~Lemma~\ref{lem:induced}, (\ref{char. map}) and (\ref{gen.fun.qr}), we have
\begin{align}
\ds \sum_{j}t^j\text{ch}(\mathcal{C}_n\otimes S^jV)
&=\sum_jt^j\sum_{\nu\models n, n(\nu)=j}
q_{\nu}(z)\notag\\
&=\sum_{\nu\models n}
\prod _{r\geq 0}q_{\nu_{r+1}}(z)(t^{r})^{\nu_{r+1}}\notag \\
&=[u^n] \prod_{r\geq 0}\sum_{s\geq 0}q_s(z)(t^{r}u)^s\notag\\
&=[u^n] \prod_{i\geq 1,r\geq 0}\frac{1+z_it^ru}{1-z_it^ru}.
 \notag
\end{align}
Recall the notation $Q_{\lambda}(t^{\bullet})$ from the introduction.
It follows from~(\ref{Cauchy identity}) that
\begin{align}
\ds \prod_{i\geq 1,j\geq 0}\frac{1+z_it^ju}{1-z_it^ju}
=\sum_{\lambda:\text{ strict}}2^{-\ell(\lambda)}u^{|\lambda|}Q_{\lambda}(t^{\bullet})Q_{\lambda}(z).\notag
\end{align}
Hence
\begin{align}
\ds \sum_{j}t^j\text{ch}(\mathcal{C}_n\otimes S^jV)=
\sum_{\lambda\vdash_s
n}2^{-\ell(\lambda)}Q_{\lambda}(t^{\bullet})Q_{\lambda}(z).\notag
\end{align}
Since the characteristic map ch is an isomorphism, we have an isomorphism of $\HC$-modules:
\begin{align*}
\mathcal{C}_n\otimes S^* V\cong\bigoplus_{\lambda\vdash_s n}
2^{-\frac{\delta(\lambda)+\ell(\lambda)}{2}}Q_{\lambda}(t^{\bullet})
D^{\lambda}.
\end{align*}
This together with Theorem~B implies Theorem~A.
\end{proof}

The following corollary follows directly from Theorem~A
and Proposition~\ref{prop:mult.equiv}.
\begin{cor}
The graded multiplicity of $D^\la_-$ in the $\spin$-module
$\mathcal{B}_n\otimes S^* V$ is given by \eqref{hook} unless $n$ is
odd and $\ell(\la)$ is even; in this case, the graded multiplicity
is
$$
2^{-\frac{\ell(\lambda)}{2}-1}
\frac{t^{n(\lambda)}\prod_{(i,j)\in \la^*}(1+t^{c_{ij}})}
{\prod_{(i,j)\in \la^*}(1-t^{h^*_{ij}})}.
$$
\end{cor}

\subsection{A graded regular $\HC$-module}

It is well known that the algebra $ (S^* V)^{S_n}$ of
$S_n$-invariants in $S^* V$ is a free polynomial algebra
whose Hilbert series $P(t)$ is given by
\begin{align}
\ds P(t)=\frac{1}{(1-t)(1-t^2)\cdots(1-t^n)}.\label{Poincare series}
\end{align}
Define the ring of coinvariants $(S^* V)_{S_n}$ to be the
quotient of $ S^* V$ by the ideal generated by the homogeneous invariant
polynomials of positive degrees.
It is well known that $S^* V$ is a free module over the algebra $(S^* V)^{S_n}$,
 and so we have an isomorphism of graded $\C S_n$-modules 
\begin{align}
 S^* V&\cong (S^* V)_{S_n}
\otimes_{\C} (S^* V)^{S_n}.
\label{coin.tensor inv.}
\end{align}

\begin{thm}\label{Cliff.tensor coin.}
The graded multiplicity of $D^{\la}$ in $\mathcal{C}_n\otimes (S^* V)_{S_n}$ is
\begin{align*}
\ds 2^{-\frac{\ell(\la)+\delta(\la)}{2}}
\frac{t^{n(\la)}(1-t)(1-t^2)\cdots(1-t^n)\prod_{(i,j)\in
\la^*}(1+t^{c_{ij}})}{\prod_{(i,j)\in \la^*}(1-t^{h^*_{ij}})}.
\end{align*}
\end{thm}

\begin{proof}
Follows directly from Theorem~A, (\ref{Poincare series}), and
\eqref{coin.tensor inv.}.
\end{proof}

\begin{rem}
Recall the isomorphism of $\HC$-modules $D^{(n)} \cong \Cl$ (cf. \cite{Kle}).
It follows that the graded multiplicity
of $\Cl$ in $\mathcal{C}_n\otimes (S^* V)_{S_n}$ is
$
(1+t)(1+t^2)\cdots(1+t^{n-1}),
$
and  the graded multiplicity
of $\Cl$ in $\mathcal{C}_n\otimes S^*
V$ is
\begin{align}  \label{basic mult}
\frac{(1+t)(1+t^2)\cdots(1+t^{n-1})}{(1-t)(1-t^2)\cdots(1-t^n)}.
\end{align}
\end{rem}

\begin{rem}
The number $g^{\la}$ of standard shifted Young tableaux of shape
$\la$ is known to be (cf. \cite{Sa2, Mac})
\begin{align}
g^{\la}=\frac{n!}{\prod_{(i,j)\in \la^*}h^*_{ij}}.\notag
\end{align}
By the isomorphism of $\C S_n$-modules
$(S^* V)_{S_n} \cong\C S_n$, the $\HC$-module
$\mathcal{C}_n\otimes (S^* V)_{S_n}$ is isomorphic to the regular representation of $\HC$. It is
known (cf. \cite{Kle}) that the multiplicity of $D^{\la}$ in the regular
representation of $\HC$ is given by
$\frac{1}{2^{\delta(\la)}} \ {\rm dim} D^{\la}$. By
specializing $t=1$ in Theorem~\ref{Cliff.tensor coin.},
we recover the dimension formula
$
{\rm dim} D^{\la}=2^{n-\frac{\ell(\la)-\delta(\la)}{2}} g^\la.
$
\end{rem}

\section{The graded multiplicity in $\mathcal{C}_n\otimes S^* V\otimes\wedge^*
V$} \label{sec:bigraded}
\subsection{The $S_n$-module $ S^* V\otimes\wedge^* V$}

The $S_n$-action on $V=\C^n$ induces a natural $S_n$-action on the exterior algebra
$$
\wedge^{*}V
=\bigoplus_{q=0}^n \wedge^{q}V.
$$
This gives rise to a $\Z_+ \times \Z_+$ bi-graded $\C S_n$-module structure on
\begin{align*}
S^*V\otimes\wedge^*V=\bigoplus_{p\geq 0, 0\leq
q\leq n}S^pV\otimes\wedge^qV.
\end{align*}

According to Kirillov and Pak \cite{KP}, the bi-graded multiplicity of the Specht module
$S^{\la}$ for $\la\vdash n$ in $S^*V\otimes
\wedge^*V$ is given by
\begin{align}  \label{eq:KPak}
\sum_{p=0}^\infty \sum_{q=0}^n t^p s^q m_\la (S^pV \otimes \wedge^qV)
= \frac{\prod_{(i,j)\in\la}(t^{i-1}+st^{j-1})}{\prod_{(i,j)\in\la}(1-t^{h_{ij}})},
\end{align}
which can be rewritten as
\begin{align*}
 \frac{t^{n(\la)}\prod_{(i,j)\in\la}(1+st^{c_{ij}})}{\prod_{(i,j)\in\la}(1-t^{h_{ij}})}.
\end{align*}
In particular, this recovers Solomon's formula  \cite{So}
for the generating function for the bi-graded $S_n$-invariants
in $S^*V\otimes
\wedge^*V$:
\begin{align}   \label{Solomon}
 \frac{(1+s)(1+st)\cdots(1+st^{n-1})}{(1-t)(1-t^2)\cdots(1-t^n)}.
\end{align}
The formal similarity between the graded multiplicities \eqref{hook} and \eqref{eq:KPak}
 in very different settings is rather striking.
Also compare the similarity between \eqref{basic mult} and \eqref{Solomon}.
\subsection{Proof of Theorem~C}
\label{subsec:Cliff.Sym.Exter}

\begin{lem}\label{lem:cliff.sgn}
The following holds
as $\HC$-modules:
\begin{align*}
{\rm ind}^{\HC}_{\mathbb{C}S_n}{\rm sgn}\cong {\rm
ind}^{\HC}_{\mathbb{C}S_n}{\bf 1}.
\end{align*}
\end{lem}

\begin{proof}
Define a $\C$-linear map
\begin{align*}
f: {\rm ind}^{\HC}_{\mathbb{C}S_n}{\rm sgn}=\Cl\otimes {\rm sgn}&\rightarrow
{\rm ind}^{\HC}_{\mathbb{C}S_n}{\bf 1}\notag\\
c\otimes 1&\mapsto c\cdot(c_1c_2\cdots c_n \otimes 1)
\end{align*}
It is straightforward to show that $f$ is actually a $\HC$-module isomorphism.
\end{proof}

\begin{lem} \label{Cliff.Epq}
For $p\geq 0, 0\leq q\leq n$, as $\HC$-modules, we have
\begin{align*}
\Cl \otimes S^{p}V\otimes\wedge^qV\cong
\bigoplus_{\alpha,\beta}{\rm ind}^{\HC}_{\C
(S_{\alpha}\times S_{\beta})}\mathbf{1},
\end{align*}
summed over all
$\alpha\models n-q, \beta\models q$
with $n(\alpha)+n(\beta)=p$.
\end{lem}

\begin{proof}
Arguing similarly as in the proof of Theorem~A,
we have an isomorphism of $\mathbb{C}S_n$-modules:
\begin{align*}
S^pV\otimes\wedge^qV
\cong\bigoplus_{\alpha,\beta}{\rm ind}^{\C S_n}_{\C
(S_{\alpha}\times S_{\beta})} (\mathbf{1}\otimes{\rm sgn}),
\end{align*}
summed over all
$\alpha\models n-q, \beta\models q$
with $n(\alpha)+n(\beta)=p$.
Now the lemma follows by applying $\text{ind}_{\C S_n}^{\HC}$ to the above isomorphism and
using Lemma~\ref{lem:cliff.sgn}.
\end{proof}

We are ready to prove Theorem C from the Introduction.

\begin{proof}[Proof of Theorem~C]
It follows by~(\ref{char. map}) and Lemma~\ref{Cliff.Epq} that
\begin{align*}
{\rm ch}(\Cl \otimes  S^{p}V\otimes\wedge^qV)
=\sum_{\alpha,\beta}q_{\alpha}(z)q_{\beta}(z),
\end{align*}
summed over all
$\alpha=(\alpha_1, \alpha_2,\ldots) \models n-q,
\beta  =(\beta_1,\beta_2,\ldots) \models q$
with $n(\alpha)+n(\beta)=p$.
Hence,
\begin{align*}
\sum_{p\geq 0,0\leq q\leq n} & t^ps^q{\rm ch}(\mathcal{C}_n\otimes
S^{p}V\otimes\wedge^qV)  \label{eqn:Cliff.sym exterior} \\
&=\sum_{p\geq 0,0\leq q\leq n}t^ps^q
\sum_{\alpha\models n-q,\beta\models q, n(\alpha)+n(\beta)=p}q_{\alpha}(z)q_{\beta}(z)\notag\\
&=[u^n]\sum_{\alpha_1, \alpha_2,\ldots,\beta_1,\beta_2,\ldots}
\prod _{r\geq 0}q_{\alpha_{r+1}}(z)(t^{r}u)^{\alpha_{r+1}}
\prod _{k\geq 0}q_{\beta_{k+1}}(z)(t^{k}su)^{\beta_{k+1}}\notag \\
&=[u^n] \prod_{i\geq 1,r\geq 0} \frac{1+z_it^ru}{1-z_it^ru}
 \frac{1+z_it^rsu}{1-z_it^rsu}  \notag\\
&=\sum_{\lambda\vdash_s n}2^{-\ell(\la)}
Q_{\la}(t^{\bullet};st^{\bullet})Q_{\la}(z) \nonumber
\end{align*}
where we have used the short-hand notation $Q_{\la}(t^{\bullet};st^{\bullet})$
from Introduction and \eqref{Cauchy identity} in the last equation.
Theorem~C follows.
\end{proof}

\begin{rem}
It is an interesting open problem to find an explicit formula for
$Q_{\lambda}(t^{\bullet};st^{\bullet})$. Consider a
Koszul $\Z_+$-grading which counts the standard generators of
$S^*V$ as degree $2$ and the standard generators of
$\wedge^*V$ as degree $1$. This corresponds precise
to setting $t=s^2$, and hence,
$Q_{\lambda}(t^{\bullet};st^{\bullet})=Q_{\lambda}(s^{\bullet})$.
Therefore, for the Koszul grading, the graded multiplicity of
$D^\la$ in $\mathcal{C}_n\otimes S^*V\otimes
\wedge^*V$ is given by the same formula \eqref{eq:bi mult}.
\end{rem}
\subsection{Some consequences of Theorem~C}

Recall the notation of $Q_{\la}(t^{\bullet};st^{\bullet})$ from Introduction.
By Proposition~\ref{prop:mult.equiv},
we have the following corollary as a counterpart of Theorem~C for $\spin$.
\begin{cor}
The bi-graded multiplicity of $D^\la_-$ in the $\spin$-module
$\mathcal{B}_n\otimes S^*V\otimes\wedge^*V$
is given by \eqref{eq:bi mult} unless $n$ is odd and $\ell(\la)$ is even;
in this case, the bi-graded multiplicity is
$$
2^{-\frac{\ell(\lambda)}{2}-1}Q_{\la}(t^{\bullet};st^{\bullet}).
$$
\end{cor}

\begin{cor}   \label{mult wedge}
The graded multiplicity of $D^{\la}$ in the $\HC$-module
${\rm ind}^{\HC}_{\C
S_n}(\wedge^*V)$ is given by $ \ds
2^{-\frac{\ell(\lambda)+\delta(\lambda)}{2}} Q_{\lambda}(1,s). $
Moreover,
\begin{align}
\ds Q_{\lambda}(1,s)= \left \{
 \begin{array}{ll}
\ds \frac{2^{\ell(\la)}(1+s)(s^l-s^k)}{1-s},
   & \text{ if }\lambda=(k,l) \text{ with }k>l\geq 0,\\
0, \text{ otherwise.}
 \end{array}\notag
 \right.
\end{align}
\end{cor}
\begin{proof}
The first statement is obtained by setting $t=0$ in Theorem~C.
By~(\ref{SchurQ1}), we see that
\begin{align}
\ds Q_{\lambda}(z_1,z_2)= \left \{
 \begin{array}{ll}
\ds \frac{2^{\ell(\la)}(z_1+z_2)(z_1^kz_2^l-z_1^lz_2^k)}{z_1-z_2},
   & \text{ if }\lambda=(k,l) \text{ with }k>l\geq 0,\\
0, \text{ otherwise.}
 \end{array}\notag
 \right.
\end{align}
The Corollary follows by setting $z_1=1$ and $z_2=s$.
\end{proof}

Setting $s=0$ in Theorem~C,
we recover  Theorem~A.

\end{document}